\documentclass[12pt]{article}
\usepackage[left=3cm,right=3cm, top=2.5cm,bottom=2.5cm,bindingoffset=0cm]{geometry}
\usepackage{cmap}

\usepackage{amsmath,amsthm,amssymb}
\usepackage[T2A]{fontenc}
\usepackage[english]{babel}
\usepackage[cp1251]{inputenc}
\usepackage{amscd}

\usepackage{graphicx}

\usepackage{caption}
\usepackage{subcaption}

 \usepackage{cite}

\usepackage{hyperref}

\newtheorem{definition}{Definition}
\newtheorem{remark}{Remark}
\newtheorem{theorem}{Theorem}
\newtheorem{corollary}{Corollary}
\newtheorem{lemma}{Lemma}
\newtheorem{assertion}{Assertion}

\newcommand{\eLie}{\mathrm{e}(3)}
\newcommand{\const}{\mathrm{const}}
\newcommand{\bR}{\textbf{R}}
\newcommand{\bS}{\textbf{S}}
\newcommand{\bv}{\textbf{v}}
\newcommand{\cI}{\mathcal{I}}
\newcommand{\bbR}{\mathbb{R}}

\DeclareMathOperator{\sgrad}{sgrad}
\let\phi\varphi
\let\hat\widehat

\author{I.\,K.~Kozlov\thanks{No Affiliation,  E-mail: {\tt ikozlov90@gmail.com}} \quad  and \quad  A.\,A.~Oshemkov\thanks{Faculty of Mechanics and Mathematics, Moscow State University,
Moscow, 119991 Russia,  E-mail: {\tt a@oshemkov.ru} }}

\title{Integrable systems with linear periodic integral \\ for the Lie algebra $\eLie$}

\date{}

\begin{document}

\maketitle



\abstract{
Integrable systems with a linear periodic integral for the Lie algebra $\eLie$ are considered.
One investigates singulariries of the Liouville foliation, bifurcation diagram of the momentum mapping, transformations of Liouville tori,
topology of isoenergy surfaces and other topological properties of such systems.}
\date{}


\begin{center}
\end{center}


\section{Introduction} \label{intro}
   
In this paper we study some topological properties of integrable Hamiltonian systems with an $S^1$-symmetry given by the Euler equations for the Lie 
algebra $\eLie$. Probably, the most well-known example of such a system is the classical Lagrange top. Roughly speaking, we consider a ``generalized'' 
Lagrange top which Hamiltonian has an arbitrary potential function and linear terms in momenta, but possesses the same $S^1$-symmetry.

We are interested in local and global topological properties of the Liouville foliation defined by the system under consideration, namely, 
the structure of bifurcation diagram and transformations of Liouville tori for critical values of the momentum mapping, non-degeneracy of equilibria 
and other singular points, the topology of isoenergy surfaces.

Note that there is a number of integrable systems with periodic linear integral which are well known in mechanics and mathematical physics,
which phase topology were studied by various authors. In particular, there are Lagrange and Kirchhoff integrable cases in rigid body dynamics 
(for the description of their topology see \cite{BF, BM, Osh91}),
the integrable case of Leggett equations describing dynamics of spin in the superfluid $^3$He (the bifurcation diagram and Fomenko invariants for this
system are described in \cite{krugl}), 
the integrable case of the motion of heavy ellipsoid on a smooth horizontal plane (topological invariants for this system were found in \cite{ivoch}).
  
Topological properties of all these systems are quite similar because of an $S^1$-symmetry which imposes strong restrictions on the structure of their 
singularities. Therefore, they can be studied under a uniform scheme. In this paper we perform such an investigation for an example of Hamiltonian 
possessing a periodic linear integral on~$\eLie^*$. Note that the problem of topological investigation of
integrable systems with S1-action is discussed in paper \cite{BIKO12}, which contains a list of various open
problems in the theory of integrable systems.

Apart from the systems on $\eLie^*$ considered in this paper there are other integrable systems with $S^1$-symmetry, which were also studied by 
various authors. For instance, natural mechanical systems on surfaces of revolution homeomorphic to the sphere were studied recently 
in \cite{kant} (see also \cite{BF}). Another example is the classical Euler case in the rigid body dynamics, where the $S^1$-action is 
given not by a linear, but by a quadratic integral. The results obtained in this paper show in particular that there are some differences 
between the topological properties of the systems under consideration and other cases with an $S^1$-symmetry (for example, the one investigated 
in \cite{kant} or the Euler case).

The article is organized as follows. In Section 2 we describe the systems under consideration. We start the analysis with the study of non-deneracy and types
of singular points of rank 0 in Section 3 (Corollary~\ref{Cor:Rank0Points} ). In Section 4 we find singular points of rank 1 
(Theorem \ref{Th:CriticalCircle}) and describe the bifurcation diagrams of the system (Theorems \ref{Th:BifDiag0} and~\ref{Th:BifDiag}). 
In Section 5 we determine types of non-degenerate points of rank~1 (Theorem~\ref{Th:Rank1Type}) and specify the corresponding Liouville tori bifurcations 
(Theorem~\ref{Th:Atoms}). Finally, in Section 6 we list all possible isoenergy surfaces for the system (Theorem~\ref{Th:Q3}).
\section{Description of the system} \label{system}
   
Let us recall that the Lie--Poisson bracket for the Lie algebra $\eLie$ is given by the formulas
\begin{equation} \label{Eq:E3PoissBracket} 
\{S_i, S_j\} = \varepsilon_{ijk}S_k, \quad \{S_i, R_j\} = \varepsilon_{ijk}R_k, \quad \{R_i, R_j\} = 0,
\end{equation}  
where  $S_1,S_2,S_3,R_1,R_2,R_3$ are linear coordinates on the dual space $\eLie^*$ for the Lie algebra~$\eLie$. 
We will use the notation $\bS=(S_1,S_2,S_3)$ and $\bR=(R_1,R_2,R_3)$ and also $\langle{\cdot}{,}{\cdot}\rangle$ and ${\,}\times{\,}$
for the scalar and vector product of 3-dimensional vectors.

A Hamiltonian system with Hamiltonian $H$ is given by the Euler equations 
\begin{equation*}
\dot{x}_i = \{ x_i, H\},
\end{equation*}
which for the Lie algebra $\eLie$ take the form
\begin{equation*}
\dot{\bS}=\frac{\partial H}{\partial\bS}\times\bS+\frac{\partial H}{\partial\bR}\times\bR,\qquad 
\dot{\bR}=\frac{\partial H}{\partial\bS}\times\bR.
\end{equation*}

Bracket \eqref{Eq:E3PoissBracket} has two Casimir functions:  
\[F_1 =\langle\bR,\bR\rangle, \qquad F_2 =\langle\bS,\bR\rangle.\]
Their regular common level surfaces 
\begin{equation} \label{Eq:LevelSurf} 
M^4_{a,g} = \{ (\bS, \bR)\mid\,  F_1 (\bS,\bR) = a,\, F_2 (\bS, \bR) = g, \}, \qquad a>0,
\end{equation} 
are the sympectic leaves of bracket \eqref{Eq:E3PoissBracket} and are the orbits of the coadjoint repsresentation for the Lie algebra $\eLie$. We are interested in integrable Hamiltonian systems on the orbits $M^4_{a,g}$ for which some linear function on $\eLie^*$ is a first integral
defining an $S^1$-action. 

Let us describe several examples of such systems from mechanics and mathematical physics, which are integrable cases of the Euler equations 
for the Lie algebra $\eLie$ with Hamiltonian $H$ and integral $K$ 
(an explanation of physical sense for parameters and variables of these systems can be found in \cite{BF, BM, krugl, ivoch}).

1) The Lagrange case. 
This is a symmetric top with two equal moments of inertia which center of gravity lies on the symmetry axis:
$$
H=\frac{S_1^2}A+\frac{S_2^2}A+\frac{S_3^2}B-pR_3,\quad K=S_3,\quad\text{where }A,B,p=\const.
$$

2) The Kirchhoff case.
This system describes the motion of a dynamically symmetric rigid body in an ideal fluid:
$$
\begin{aligned}
H&=AS_1^2+AS_2^2+aS_3^2+2(BS_1R_1+2BS_2R_2+bS_3R_3)+\\&+CR_1^2+CR_2^2+cR_3^2,\quad K=S_3,\quad\text{where } A,a,B,b,C,c=\const.
\end{aligned}
$$

3) The following integrable case for the Leggett system describing the dynamics of spin in the superfluid $^3$He:
$$
H=S_1^2+S_2^2+S_3^2-\gamma S_3-R_3^2,\quad K=S_3,\quad\text{where }\gamma=\const.
$$

4) Integrable system describing the motion of a dynamically and geometrically symmetric heavy ellipsoid on a smooth horizontal plane:
$$
\begin{aligned}
H&=\frac{S_1^2+S_2^2+A(S_1R_1+S_2R_2)^2}{2b(1+A(R_1^2+R_2^2))}+\frac{S_3^2}{2J}+\sqrt{1+cR_3^2}+sR_3,\\
K&=S_3\quad\text{where }A=\frac{cR_3^2}{1+cR_3^2},\quad b,c,J,s=\const.
\end{aligned}
$$

In all these examples the additional integral is the function $S_3$ on $\eLie^*$. 
Let us explain that this is a general case if we require that the integral is linear and periodic.

\begin{assertion}
Let $K$ be a linear functions on $\eLie^*$ which Hamiltonian flow $\sgrad K$ defined by bracket {\rm\eqref{Eq:E3PoissBracket}}
is periodic. Then there is a linear change of variables preserving the bracket {\rm\eqref{Eq:E3PoissBracket}} taking the function $K$
to $cS_3$, where $c$ is some constant.
\end{assertion}

\begin{proof} 
Let $K=\alpha_1S_1+\alpha_2S_2+\alpha_3S_3+\beta_1R_1+\beta_2R_2+\beta_3R_3$. 
For an arbitrary orthogonal matrix $A$ the transformation $\Phi_A:(\bS,\bR)\to(A\bS,A\bR)$ preserves bracket (\ref{Eq:E3PoissBracket}).
If $\alpha_1=\alpha_2=\alpha_3=0$, then we can choose a matrix $A$ such that $\Phi_A$ takes the function $K$ to $\lambda R_3$, where $\lambda=\const$.
It is clear that the Hamiltonian flow of the function $\lambda R_3$ is not periodic, since the trajectories of the field
$\sgrad R_3=(-R_2,R_1,0,0,0,0)$ are straight lines in $\eLie^*$.

If there are non-zero $\alpha_i$, then applying an appropriate transformation $\Phi_A$ we can transform $K$ to a function
of the form $cS_3+\beta'_1R_1+\beta'_2R_2+\beta'_3R_3$.
It is easy to check that for any vector $\bv$ the transformations $\Psi_{\bv}:(\bS,\bR)\to(\bS+\bv\times\bR,\bR)$ also preserve 
bracket (\ref{Eq:E3PoissBracket}). This allows one to transform the function $K$ to the form $cS_3+\lambda R_3$, where $c\ne0$.

Now consider the function $K=S_3+\lambda R_3$ and determine for which $\lambda$ the Hamiltonian flow of $K$ is periodic.
Integral trajectories for the field $\sgrad K=(-S_2-\lambda R_2,S_1+\lambda R_1,0,-R_2,R_1,0)$ can be explicitly written:
$$
\begin{aligned}
\gamma(t)=
(&(s_1{-}\lambda r_2t)\cos t{-}(s_2{+}\lambda r_1t)\sin t,\,(s_2{+}\lambda r_1t)\cos t{+}(s_1{-}\lambda r_2t)\sin t,\\
 &s_3,\,r_1\cos t-r_2\sin t,\,r_2\cos t+r_1\sin t,\,r_3),
\end{aligned}
$$
where $s_1,s_2,s_3,r_1,r_2,r_3$ are constants. It is clear from this formula that the trajectories are periodic only for $\lambda=0$.
\end{proof} 

\begin{remark} 
It is well known that an action of any compact group can be linearized at a fixed point and that for an action of the circle
$S^1$ the corresponding tangent space can be represented as a sum of invariant two-dimensional subspaces. Thus among all linear functions on $\eLie^*$ 
the periodic integrals are distiguished by the property that their linearization at any singular point is a unitary operator with respect to a 
complex structure on the tangent space. It also follows that up to the choice of the coordinate system and multipltication by a constant any periodic 
linear integral on~$\eLie^*$ is~$S_3$.                                    
\end{remark} 

Further we will consider Hamiltonian systems for the Lie algebra $\eLie$ which possess the first integral $K=S_3$ and which Hamiltonian $H$ 
is quadratic in $S$, i.e., 
\begin{equation}  \label{Eq:HamGen} 
H=A_1S_1^2+A_2S_2^2+A_3S_3^2+f_1(\bR)S_1+f_2(\bR)S_2+f_3(\bR)S_3+f_4(\bR),  
\end{equation}
where $A_1,A_2,A_3$ are arbitrary positive constants and $f_1,f_2,f_3,f_4$ are smooth functions of $R_1,R_2,R_3$.

First of all, let us rewrite Hamiltonian \eqref{Eq:HamGen} in a more convient way using its commutativity with the function $S_3$.

\begin{assertion} 
Up to multiplication by a constant any Hamiltonian of the form \eqref{Eq:HamGen} commuting with the function $K=S_3$  
has the form 
\begin{equation} \label{Eq:H}  
\begin{aligned}  
H&=\frac12\Bigl(S_1^2+S_2^2+\frac{S_3^2}\beta\Bigr)+g_1(\bR^2,R_3)(S_1R_2-S_2R_1)+\\  
 &+g_2(\bR^2,R_3)\langle\bS,\bR\rangle+g_3(\bR^2,R_3)S_3+V(\bR^2, R_3), 
\end{aligned} 
\end{equation}   
where $\beta>0$ and the functions $g_1,g_2,g_3,V$ depend only on $\bR^2$ and $R_3$ and are smooth if $\bR^2\ne0$.
\end{assertion}

\begin{proof} 
The Hamiltonian vector field for the function $K$ is equal to 
$$
\sgrad K=-R_2\frac\partial{\partial R_1}+R_1\frac\partial{\partial R_2}-S_2\frac\partial{\partial S_1}+S_1\frac\partial{\partial S_2}.
$$

Since $\{H,K\}=(\sgrad K)H=0$, we get 
$$
\begin{gathered}
(\sgrad K)H=2(A_2-A_1)S_1S_2+\\
+\Bigl({-}R_2\frac{\partial f_1}{\partial R_1}{+}R_1\frac{\partial f_1}{\partial R_2}{+}f_2(\bR)\Bigr)S_1+ 
 \Bigl({-}R_2\frac{\partial f_2}{\partial R_1}{+}R_1\frac{\partial f_2}{\partial R_2}{-}f_1(\bR)\Bigr)S_2+\\
+\Bigl(-R_2\frac{\partial f_3}{\partial R_1}+R_1\frac{\partial f_3}{\partial R_2}\Bigr)S_3+
 \Bigl(-R_2\frac{\partial f_4}{\partial R_1}+R_1\frac{\partial f_4}{\partial R_2}\Bigr)=0.
\end{gathered}
$$
Hence, $A_1=A_2$ (multiplying by a constant we can make both these constants equal to $\frac12$) and the four expressions in the brackets are equal to zero.

In polar coordinates $(\rho, \varphi) $ on the plane $(R_1, R_2)$ the vector field $\frac{\partial}{\partial \varphi}$ is exactly
$ -R_2\frac{\partial }{\partial R_1}+ R_1\frac{\partial }{\partial R_2}$.
Therefore, 
$$
\frac{\partial f_3}{\partial \varphi} =0,\quad\frac{\partial f_4}{\partial \varphi} =0,\qquad
\frac{\partial f_1}{\partial \varphi} =-f_2,\quad\frac{\partial f_2}{\partial \varphi} =f_1.
$$
The first two of these equations imply that $f_3$ and $f_4$ depend only on $\rho$ and $R_3$ or, equivalently, 
$f_3(\bR) = g_3(\bR^2, R_3)$ and $f_4(\bR) = V(\bR^2, R_3)$. 
The latter two equations can be cosidered as a system of ODE with parameters $\rho$ and $R_3$.
Solving it, we obtain
$$
\begin{aligned} 
f_1 &= f_{11}(\rho,R_3)\cos\varphi + f_{12}(\rho,R_3)\sin\varphi=\frac{ f_{11}(\rho,R_3)}\rho R_1+\frac{f_{12}(\rho,R_3)}\rho R_2,\\
f_2 &=-f_{12}(\rho,R_3)\cos\varphi{+}f_{11}(\rho,R_3)\sin\varphi=\frac{-f_{12}(\rho,R_3)}\rho R_1{+}\frac{f_{11}(\rho,R_3)}\rho R_2.
\end{aligned}
$$
Since $\rho=\sqrt{R_1^2+R_2^2}$ we get the desired form for the Hamiltonian~$H$.
\end{proof}

\section{Singularities of rank 0} \label{S:CritPointsRank0} 

It turns out that equilibria points for a Hamiltonian system on $\eLie^*$ possessing a linear periodic integral $K$ are exactly 
the points where $\sgrad K=0$. This gives the following simple description for singularities of rank 0 of such integrable Hamiltonian systems
(not necessarily with Hamiltonian of the form \eqref{Eq:HamGen}).

\begin{theorem} \label{A:Rank0Points} 
The set of singular points of rank $0$ for an integrable Hamiltonian system on $\eLie^*$ with arbitrary Hamiltonian $H$ possessing 
the integral $K=S_3$ is the two-dimensional subspace 
\begin{equation} \label{Eq:Rank0Points} 
\{(0,0,S_3,0,0,R_3)\} 
\end{equation} 
in $\eLie^*$. In particular, for each orbit $M^4_{a,g}$ there are precisely two singular points of rank $0${\rm:} 
$$
\Bigl(0,0,\pm\frac g{\sqrt a},0,0,\pm\sqrt a\Bigr).
$$
\end{theorem}

\begin{proof} 
The Hamiltonian vector field of a function $f$ on $\eLie^*$ has the form 
\begin{equation} \label{sgrad}
\sgrad f=\Bigl(\frac{\partial f}{\partial\bS}\times\bS+\frac{\partial f}{\partial\bR}\times\bR,\frac{\partial f}{\partial\bS}\times\bR\Bigr),
\end{equation}
and for the function $K=S_3$ we have $\sgrad K=(-S_2,S_1,0,-R_2,R_1,0)$. Therefore, $\sgrad K=0$ exactly at points \eqref{Eq:Rank0Points}. 
Thus, points other than \eqref{Eq:Rank0Points} can not be singular points of rank $0$. 

Let us prove that $\sgrad H$ vanishes at points \eqref{Eq:Rank0Points}.
The functions $H$ and $K$ commute with respect to bracket \eqref{Eq:E3PoissBracket}, i.e.,
$d_yH(\sgrad_yK)=0$ for any point $y\in\eLie^*$ (the index $y$ in $d_yf$ or $\sgrad_yf$ denotes the point at which 
the differential or, respectively, skew-gradient of the function $f$ is taken). 
Taking the differential of the function $d_yH(\sgrad_yK)$ at any point $y=(0,0,S_3,0,0,R_3)$, we get 
\begin{equation} \label{LKdH}
A_K^*(d_yH)=0, 
\end{equation}
where $A_K$ is the linearization operator for the vector field $\sgrad K$ at the point $y$, since $\sgrad_yK=0$.
The matrix of the operator $A_K:\eLie^*\to\eLie^*$ has the form
\[\begin{pmatrix}
0&-1&0&0& 0&0\\
1& 0&0&0& 0&0\\
0& 0&0&0& 0&0\\
0& 0&0&0&-1&0\\
0& 0&0&1& 0&0\\
0& 0&0&0& 0&0
\end{pmatrix}\] 
and therefore condition \eqref{LKdH} implies that 
$\frac{\partial H}{\partial S_1}=\frac{\partial H}{\partial S_2}=\frac{\partial H}{\partial R_1}=\frac{\partial H}{\partial R_2}=0$
at any point $y=(0,0,S_3,0,0,R_3)$. 
Hence $\sgrad H$ vanishes at points \eqref{Eq:Rank0Points}, since at a point $y=(0,0,S_3,0,0,R_3)$  formula \eqref{sgrad} becomes
$$
\sgrad_yf=\Bigl(S_3\frac{\partial f}{\partial S_2}+R_3\frac{\partial f}{\partial R_2},-S_3\frac{\partial f}{\partial S_1}-R_3\frac{\partial f}{\partial R_1},
0,R_3\frac{\partial f}{\partial S_2},-R_3\frac{\partial f}{\partial S_1},0\Bigr).
$$

Theorem \ref{A:Rank0Points} is proved. 
\end{proof}

Now, let us state when these zero-rank points are non-degenerate and determine their type
(for more information about non-degeneracy of singular points of a momentum mapping see \cite{BF}).

\begin{theorem} \label{Th:Rank0Points} 
For an integrable Hamiltonian system on $\eLie^*$ with arbitrary Hamiltonian $H$ possessing the integral $K=S_3$, the singular point of rank $0$ 
$$
P_\pm=\Bigl(0,0,\pm\frac g{\sqrt a},0,0,\pm\sqrt a\Bigr)
$$
on the orbit $M^4_{a,g}$ is non-degenerate iff $q\ne0$, where  
\begin{align} 
\label{Eq:SpectrumQ}  
q&=p^2+R_3^2(H_{11}H_{22}-|H_{12}|^2),\\
\label{Eq:SpectrumP}  
p&=\frac g{2R_3}\frac{\partial^2H}{\partial S_1^2}+R_3\frac{\partial^2 H}{\partial S_1\partial R_1}-\frac{\partial H}{\partial S_3}, 
\end{align} 
and 
\begin{equation*} 
\begin{gathered} 
H_{11}=\frac{\partial^2 H}{\partial S_1^2},\qquad 
H_{12}=\Bigl(\frac{\partial^2 H}{\partial S_1\partial R_1}-\frac1{R_3}\frac{\partial H}{\partial S_3}\Bigr)+i\frac{\partial^2 H}{\partial S_2\partial R_1},\\
H_{22}=\frac{\partial^2 H}{\partial R_1^2}+\frac g{R_3^3}\frac{\partial H}{\partial S_3}-\frac1{R_3}\frac{\partial H}{\partial R_3}.
\end{gathered} 
\end{equation*} 

Also, if the point $P_\pm$ is non-degenerate, then its type is
\begin{enumerate}
\item center-center if $q>0$,
\item focus-focus if $q<0$.
\end{enumerate}
\end{theorem}

Theorem \ref{Th:Rank0Points} holds for any Hamiltonian $H$ that commutes (and is functionally independent) with $K=S_3$. 
For the Hamiltonian $H$ quadratic in $\bS$ the condition of non-degeneracy and types of singular points of rank~0 are as follows.

\begin{corollary} \label{Cor:Rank0Points} 
For Hamiltonian $\eqref{Eq:H}$ the type of singular points of rank~$0$ is completely determined as in Theorem~{\rm\ref{Th:Rank0Points}} by 
$$
q=\frac{g^2}{4R_3^2}-R_3^2g_1^2(a,R_3)+gR_3\frac{\partial g_2}{\partial R_3}(a,R_3)-g\frac{\partial g_3}{\partial R_3}(a,R_3)-
R_3\frac{\partial V}{\partial R_3}(a,R_3).
$$
\end{corollary}

\begin{proof}
Calculating all expressions from Theorem~\ref{Th:Rank0Points}, we have 
\begin{equation} \label{Eq:Hij} 
\begin{gathered} 
H_{11}=1,\qquad H_{12}=-\frac1{R_3}\Bigl(\frac g{\beta R_3}+g_3(a,R_3)\Bigr)-ig_1(a,R_3),\\
H_{22}=\frac g{R^3_3}\Bigl(\frac g{\beta R_3}+g_3(a,R_3)\Bigr)-\\ 
-\frac1{R_3}\Bigl(g\frac{\partial g_2}{\partial R_3}(a,R_3)+\frac g{R_3}\frac{\partial g_3}{\partial R_3}(a,R_3)+\frac{\partial V}{\partial R_3}(a,R_3)\Bigr),
\end{gathered} 
\end{equation} 
and 
\begin{equation*} 
p=\frac g{R_3}\Bigl(\frac12-\frac1\beta\Bigr)-g_3(a,R_3).  
\end{equation*} 
Substituting them into~\eqref{Eq:SpectrumQ}, one obtains the required formula for $q$.
\end{proof}

In order to prove Theorem~\ref{Th:Rank0Points} we use the following criteria of non-degeneracy (see~\cite{BF}), which can be regarded as a definition.

\begin{definition} \label{Rank0NonDegenCrit} 
A point $P$ of rank $0$ for an integrable Hamiltonian system with Hamiltonian $H$ and integral $K$ on a symplectic manifold $M^4$ is non-degenerate 
iff the following two conditions hold{\rm:}
\begin{itemize} 
\item the linearizations $A_H$ and $A_K$ of the Hamiltonian vector fields $\sgrad H$ and $\sgrad K$ at the point $P$ are linear independent,
\item there exists a linear combination $\lambda A_H+\mu A_K$ with four different non-zero eigenvalues.
\end{itemize}
\end{definition}

Let us study the spectrum of linearization of $\sgrad H$ at the points of rank~0. Taking functions $S_1,S_2,R_1,R_2 $ as local coordinates in a neighbourhood of 
$0$-rank point $P_\pm$ on an orbit $M^4_{a,g}$ we have 
$$
R_3=\pm\sqrt{a-R_1^2-R_2^2},\qquad S_3=\frac1{R_3}(g-S_1R_1-S_2 R_2).
$$

Denote by $\hat H(S_1,S_2,R_1,R_2)$ the restriction of the fucntion $H$ onto $M^4_{a,g}$.

\begin{lemma} \label{L:SpecF} 
For any function $H$ commuting with $K=S_3$ the spectrum of the linearization operator $A_{\hat H}=\mathrm{Lin}(\sgrad\hat H)$ at the singular points $P_\pm$ 
of rank $0$ has the form $\sigma(A_{\hat H})=\{\pm i(p+\sqrt q),\pm i(p-\sqrt q)\}$, where $p$ and $q$ are given by \eqref{Eq:SpectrumP} and \eqref{Eq:SpectrumQ}. 
\end{lemma}

\begin{proof} 
In the coordinates $S_1,S_2,R_1,R_2$ the Poisson bracket on the symplectic leaf $M^4_{a, g}$ has the form 
$$
\mathcal A=\left(\begin{matrix}
0   &S_3&0   &R_3\\
-S_3&0  &-R_3&0  \\
0   &R_3&0   &0  \\
-R_3&0  &0   &0 
\end{matrix}\right).
$$
It is easy to check that the linearization of $\sgrad K$ defines a complex structure on the tangent space: 
\begin{equation} \label{Eq:AK}
A_{\hat K}=\mathrm{Lin}(\sgrad\hat K)=\left(\begin{matrix}
0&-1& 0& 0\\ 
1& 0& 0& 0\\
0& 0& 0&-1\\ 
0& 0& 1& 0
\end{matrix}\right).
\end{equation}

Since $[A_{\hat H},A_{\hat K}]=0$, the operator $A_{\hat H}$ can be complexified. The matrix of the Poisson structure can also be complexified,
i.e., we can identify $(2\times2)$-blocks $\left(\begin{smallmatrix}\alpha&-\beta\\\beta&\alpha\end{smallmatrix}\right)$ in matrices with complex 
numbers $\alpha+i\beta$. Thus, in the complex coordinates $S_1+iS_2,R_1+iR_2$ the matrix~$\mathcal A$ of the Poisson structure has the form 
$$
\mathcal A=\begin{pmatrix}-iS_3&-iR_3\\-iR_3&0\end{pmatrix}.
$$
On a symplectic manifold we have $A_{\hat H}=\mathcal A\,d^2\hat H$, and therefore $d^2\hat H$ can also be complexified. By direct calculation we get
$$
d^2\hat H=\begin{pmatrix}H_{11}&H_{12}\\ \,\overline{\hbox{\vphantom|$\!H\!$}}\,_{12}&H_{22}\end{pmatrix},
$$
where $H_{lj}$ are given by formulas \eqref{Eq:Hij}. The imaginary parts of $H_{11}$ and $H_{22}$ vanish because $H$ commutes with $K$. 

Using the fact that if $\mu_1,\mu_2$ are eigenvalues of a matrix $(A+iB)$ for real matrices $A,B$, then the matrix 
$\left(\begin{smallmatrix}A&B\\-B&A\end{smallmatrix}\right)$ has the eigenvalues $\mu_1,\mu_2,\overline\mu_1,\overline\mu_2$, 
we obtain that the specturm of the (real) operator $A_{\hat H}$ is given by the equation
$$
\mu^2-i(S_3H_{11}+R_3H_{12}+R_3\,\overline{\hbox{\vphantom|$\!H\!$}}\,_{12})\mu+R_3^2(H_{11}H_{22}-|H_{12}|^2)=0,
$$
which solutions give the desired spectrum. Lemma \ref{L:SpecF} is proved.  
\end{proof}

\begin{remark}
It is clear from \eqref{Eq:AK} that for the integral $K=S_3$ the spectrum of the corresponding operator $A_{\hat K}$ 
is $\sigma(A_{\hat K})=\{i,-i,i,-i\}$. This doesn't immediately prove non-deneracy of points but shows that non-degenerate 
points can be only of center-center or focus-focus type.
\end{remark}

\begin{proof}[Proof of Theorem {\rm\ref{Th:Rank0Points}}]
Using  Lemma \ref{L:SpecF} and Definition \ref{Rank0NonDegenCrit} of non-degeneracy we get the condition of the theorem in all cases except 
for $q=0$ or $p^2=q$.

If $q=0$, then the spectra of $A_{\hat H}$ and $A_{\hat K}$ are proportional, thus the point is degenerate (this is precisely the moment 
when the image of a focus-focus point meets an arc of the bifurcation diagram while transforming into a center-center point).

If $p^2=q$, then the point is non-degenerate, and one should just take another linear combination with different eigenvalues 
(such a linear combination exists since the spectra of $A_{\hat H}$ and $A_{\hat K}$ are non-proportional).
\end{proof}

\section{Bifurcation diagrams}

In order to construct the bifurcation diagram let us describe all critical points of the momentum mapping. The singular points of rank $0$ are found 
in Section \ref{S:CritPointsRank0}. Thus, it remains to describe only singular points of rank $1$. The next two lemmas show that we can use some 
convenient coordinates for investigating them.

\begin{lemma} \label{L:R12zero}
For a Hamiltonian system with Hamiltonian $H$ of the form {\rm\eqref{Eq:H}} and integral $K=S_3$, the subspace $\{(\bS,\bR)\mid R_1=R_2=0\}$ in $\eLie^*$
does not contain points of rank~$1$.
\end{lemma}

\begin{proof}
Since we know all singular points of rank~0 (they are points with $R_1=R_2=S_1=S_2=0$; see Theorem~\ref{A:Rank0Points}), it suffices to prove 
that if $y=(S_1,S_2,S_3,0,0,R_3)\in\eLie^*$ is a singular point, then its coordinates $S_1$ and $S_2$ vanish. Suppose that this is not the case. 
Then $\sgrad_yK=(-S_2,S_1,0,0,0,0)\ne0$ and, therefore, $\sgrad_yH=\lambda\sgrad_yK$ for a certain $\lambda$. Hence, by formula \eqref{sgrad} 
(taking into account that $R_3\ne0$), we have $\frac{\partial H}{\partial S_1}=\frac{\partial H}{\partial S_2}=0$ at the point $y$.
But for a Hamiltonian of the form~\eqref{Eq:H} this is possible only if $S_1=S_2=0$ for the point $y$.
\end{proof}

Now, since we can assume that $R_1^2+R_2^2\ne0$, we choose new coordinates on the remaining set of points $U=\bbR^6(\bS,\bR)\setminus\{R_1=R_2=0\}$.
Note that the set $U$ is homeomorphic to $\bbR^5\times S^1$.

\begin{lemma} \label{L:coord}
Formulas
\begin{equation} \label{Eq:newcoord}
\begin{gathered}  
S_1=\frac{(g-kx)\cos\phi+m\sin\phi}{\sqrt{a-x^2}},\quad S_2=\frac{(g-kx)\sin\phi-m\cos\phi}{\sqrt{a-x^2}},\\ 
S_3=k,\quad R_1=\sqrt{a-x^2}\,\cos\phi,\quad R_2=\sqrt{a-x^2}\,\sin\phi,\quad R_3=x
\end{gathered} 
\end{equation} 
define regular coordinates $(x,m,\phi,k,a,g)$ on the set $U$, where $x^2<a$ and $\phi$ is an angular coordinate, i.e., is defined modulo $2\pi$. 

The inverse change of variables on the set $U$, i.e., the expression of $(x,m,\phi,k,a,g)$ through $(\bS,\bR)$ is as follows{\/\rm:}
\begin{equation*} \label{Eq:KIntegrals}
\begin{gathered}  
x=R_3,\quad m=M(\bS,\bR)=S_1R_2-S_2R_1,\quad\phi=\arg(R_1+iR_2),\\
k=S_3,\quad a=F_1(\bS,\bR)=\langle\bR,\bR\rangle,\quad g=F_2(\bS,\bR)=\langle\bS,\bR\rangle.
\end{gathered} 
\end{equation*} 
\end{lemma}

\begin{proof}
By direct calculation, it is easy to check that given formulas define a bijection and that the Jacobian does not vanish on $U$:
$$
\det\frac{\partial(x,m,\phi,k,a,g)}{\partial(S_1,S_2,S_3,R_1,R_2,R_3)}=2(R_1^2+R_2^2)\ne0.
$$

\kern-5.mm
\end{proof}

Substituting expressions \eqref{Eq:newcoord} into \eqref{Eq:H}, we obtain that the Hamiltonian in the coordinates $(x,m,\phi,k,a,g)$ 
on the set $U$ has the form
\begin{equation}  \label{Eq:HamiltIntegrCoord} 
H=\frac{(g{-}kx)^2{+}m^2}{2(a-x^2)}+\frac{k^2}{2\beta}+g_1(a,x)m+g_2(a,x)g+g_3(a,x)k+V(a,x). 
\end{equation}

Futher we will often write $g_1,g_2,g_3,V$ without arguments assuming that they are functions of $a$ and $x$.

The next statement describes the set of singular points of rank~1.

\begin{theorem} \label{Th:CriticalCircle} 
The set of all singular points of rank~$1$ for the system with Hamiltonian \eqref{Eq:H} and integral $K=S_3$ on $\eLie^*$ is given 
by the following two equations in the coordinates $(x,m,\phi,k,a,g)${\rm:}
\begin{gather} 
\label{Eq:Rank1Points-m} 
m=-(a-x^2)g_1, \\
\label{Eq:Rank1Points-k} 
\frac{(kx{-}g)(ka{-}gx)}{(a-x^2)^2}+xg_1^2-(a{-}x^2)g_1\frac{\partial g_1}{\partial x}+
g\frac{\partial g_2}{\partial x}+k\frac{\partial g_3}{\partial x}+\frac{\partial V}{\partial x}=0.
\end{gather} 
\end{theorem}

\begin{proof} 
Calculating the matrix of the Poisson bracket in the coordinates $(x,m,\phi,k,a,g)$, one obtains
$$
\begin{pmatrix} 
0    &a-x^2&0 &0&0&0\\
x^2-a&0    &0 &0&0&0\\
0    &0    &0 &1&0&0\\
0    &0    &-1&0&0&0\\
0    &0    &0 &0&0&0\\
0    &0    &0 &0&0&0
\end{pmatrix}.
$$
Therefore, in these coordinates the skew-gradients of $H$ and $K$ are
\begin{equation}  \label{Eq:sgradHandK} 
\begin{aligned}
\sgrad H&=\Bigl((a-x^2)\frac{\partial H}{\partial m},(x^2-a)\frac{\partial H}{\partial x},\frac{\partial H}{\partial k},0,0,0\Bigr),\\
\sgrad K&=(0,0,1,0,0,0).
\end{aligned} 
\end{equation}
Here we take into account that $\frac{\partial H}{\partial\phi}=\{H,K\}\equiv0$.

Thus the condition of linear dependence of $\sgrad H$ and $\sgrad K$ at a point $y\in\eLie^*$
$$
\sgrad H=\lambda\sgrad K
$$
is equivalent to the conditions
\begin{equation}  \label{Eq:lambda}
\frac{\partial H}{\partial m}=0,\qquad\frac{\partial H}{\partial x}=0,\qquad\frac{\partial H}{\partial k}=\lambda
\end{equation}
at the point $y$. Differentiating Hamiltonian \eqref{Eq:HamiltIntegrCoord} with respect to $m$ and $x$, we see that $\frac{\partial H}{\partial m}=0$ 
is equivalent to \eqref{Eq:Rank1Points-m} and $\frac{\partial H}{\partial x}=0$ is equivalent to \eqref{Eq:Rank1Points-k} after the substitution of $m$ 
from \eqref{Eq:Rank1Points-m}.
\end{proof} 

\begin{corollary} \label{Cor:CriticalCircle} 
On each orbit $M^4_{a,g}$ the set of singular points of rank~$1$ form a one-parameter family of critical circles, which is parametrized by points $(k,x)$ 
of curves defined by equation~\eqref{Eq:Rank1Points-k}. For each point $(k,x)$ satisfying~\eqref{Eq:Rank1Points-k} the corresponding critical circle 
in $M^4_{a,g}$ is given by the formulas
$$
\begin{gathered} 
S_1{=}\frac{(g{-}kx)\cos\phi{-}(a{-}x^2)g_1\sin\phi}{\sqrt{a-x^2}},\quad 
S_2{=}\frac{(g{-}kx)\sin\phi{+}(a{-}x^2)g_1\cos\phi}{\sqrt{a-x^2}},\\
S_3=k,\quad R_1=\sqrt{a-x^2}\,\cos\phi,\quad R_2=\sqrt{a-x^2}\,\sin\phi,\quad R_3=x,
\end{gathered}
$$
where $\phi$ is a parameter on the circle.
\end{corollary}

\begin{proof} 
As it is shown in the proof of Theorem \ref{Th:CriticalCircle}, $\sgrad K=\frac{\partial}{\partial\phi}$ in the coordinates $(x,m,\phi,k,a,g)$. Therefore, 
each critical circle is a coordinate line of the coordinate $\phi$. Substituting \eqref{Eq:Rank1Points-m} into expressions \eqref{Eq:newcoord}, we obtain 
the required formulas.
\end{proof} 

Now we can describe the bifurcation diagram. For each pair of parameters $a,g$, where $a>0$, consider the function 
\begin{equation} \label{Eq:W} 
W_{a,g}(k,x)=\frac{(g-kx)^2}{2(a-x^2)}+\frac{k^2}{2\beta}-\frac{g^2_1}2(a-x^2)+g_2g+g_3k+V,  
\end{equation} 
which is an analogue of a reduced potential. Recall that $g_1,g_2,g_3,V$ are functions of $a$ and $x$.

\begin{theorem} \label{Th:BifDiag0} 
The bifurcation diagram of the integrable Hamiltonian system with Hamiltonian $\eqref{Eq:H}$ and the integral $K=S_3$ on orbit $\eqref{Eq:LevelSurf}$ 
consists of the following subsets on the plane $\bbR^2(h,k)${\rm:}

$1)$ two points $Z_\pm$ {\rm(}they can coinside if $g=0${\rm)} with coordinates
$$
h=\frac{g^2}{2\beta a}+g\,g_2(a,\pm\sqrt a)\pm\frac g{\sqrt a}\,g_3(a,\pm\sqrt a)+V(a,\pm\sqrt a),\quad k=\pm\frac g{\sqrt a},
$$
which are the images of two singular points of rank $0${\rm;}

$2)$ the points $(h(x),k(x))$ which are the images of singular points of rank $1$ and are parametrized by the parameter $x$, 
where the function $k(x)$ is implicitly defined by the quadratic {\rm(}or linear{\rm)} equation 
$\frac{\partial W_{a,g}}{\partial x}(k,x)=0$, and $h(x)=W_{a,g}(k(x),x)$.
\end{theorem}

\begin{proof} 
The first statement immediately follows from Theorem \ref{A:Rank0Points} describing singular points of rank~0.
Similarly, the second one follows from Theorem \ref{Th:CriticalCircle} describing singular points of rank~1 by taking into account 
expression \eqref{Eq:HamiltIntegrCoord} for the Hamiltonian $H$ and definition \eqref{Eq:W} of the function~$W_{a,g}$.
\end{proof} 

\begin{remark} 
For each fixed $a,g$ the equations from Theorem {\rm\ref{Th:BifDiag0}}
\begin{equation}  \label{Eq:envelope}
h=W_{a,g}(k,x),\qquad \frac{\partial W_{a,g}}{\partial x}(k,x)=0
\end{equation}
describing the image of the set of singular points of rank~$1$ belonging to the orbit $M^4_{a,g}$ are exactly the equations for the envelope of the family 
of parabolas
\begin{equation*} \label{Eq:h(k)}
h=\Bigl(\frac{x^2}{2(a-x^2)}+\frac1{2\beta}\Bigr)k^2+B_{a,g}(x)k+C_{a,g}(x)
\end{equation*}
on the plane $\bbR^2(h,k)$ depending on the parameter $x$, where
\begin{equation} \label{Eq:BC(x)}
\begin{aligned}
B_{a,g}(x)&=g_3(a,x)-\frac{gx}{a-x^2},\\
C_{a,g}(x)&=\frac{g^2}{2(a-x^2)}-\frac{g^2_1(a,x)}2(a-x^2)+g_2(a,x)g+V(a,x)
\end{aligned}
\end{equation}
{\rm(}see formula \eqref{Eq:W}{\rm)}. 
In other words, the bifurcation diagram {\rm(}without points $Z_\pm${\rm)} can be regarded as the envelope of this family of parabolas.
\end{remark}

The bifurcation diagram $\Sigma$ is the union of $\Sigma_0=\{Z_\pm\}$ and $\Sigma_1$ which consists of the images of singular points of rank~$1$.
Let us rewrite conditions \eqref{Eq:envelope} describing $\Sigma_1$ in a more explicit parametric form.

The relation $\frac{\partial W_{a,g}}{\partial x}(k,x)=0$ from Theorem {\rm\ref{Th:BifDiag0}} is exactly equation \eqref{Eq:Rank1Points-k}.
In notation \eqref{Eq:BC(x)} it can be written as
\begin{equation} \label{Eq:k-quadratic}
\frac{ax}{(a-x^2)^2}k^2+B'_{a,g}(x)k+C'_{a,g}(x)=0,
\end{equation}
where
$$
\begin{aligned}
B'_{a,g}(x)&=\frac{\partial g_3}{\partial x}-\frac{g(a+x^2)}{(a-x^2)^2},\\
C'_{a,g}(x)&=\frac{g^2x}{(a-x^2)^2}+xg_1^2-(a-x^2)g_1\frac{\partial g_1}{\partial x}+g\frac{\partial g_2}{\partial x}+\frac{\partial V}{\partial x}.
\end{aligned}
$$
Equation \eqref{Eq:k-quadratic} is quadratic with respect to~$k$ for $x\ne0$ (it is reduced to linear equation for $x=0$). Its discriminant equals
$$
\begin{aligned}
D_{a,g}(x)
&=(B'_{a,g}(x))^2-\frac{4ax}{(a{-}x^2)^2}C'_{a,g}(x)=\frac1{(a{-}x^2)^2}\Bigl(g-(a{+}x^2)\frac{\partial g_3}{\partial x}\Bigr)^2-\\
&-\frac{4ax}{(a-x^2)^2}\Bigl(xg_1^2-(a-x^2)g_1\frac{\partial g_1}{\partial x}+g\frac{\partial g_2}{\partial x}+
x\Bigl(\frac{\partial g_3}{\partial x}\Bigr)^2+\frac{\partial V}{\partial x}\Bigr).
\end{aligned}
$$

In order to describe a parametrization of bifurcational curves consider the set
$$
\Theta_{a,g}=\{x\in\bbR\mid x^2<a,\,\,x\ne0,\,\,D_{a,g}(x)\ge0\}.
$$
Each its (arcwise) connected component is an interval, which is either non-dege\-ne\-rate (i.e., has a non-zero length) or degenerate (i.e., is a point). 
Denote the set of all non-degenerate intervals by $\cI_{a,g}$ and denote the set of degenerate intervals by $\Theta^0_{a,g}$.
Clearly, $\Theta_{a,g}\setminus\Theta^0_{a,g}=\bigcup_{I\in\cI_{a,g}}I$.

Since $\Theta_{a,g}$ is, evidently, a closed subset of $(-\sqrt a,0)\cup(0,\sqrt a)$, 
intervals from $\cI_{a,g}$ contain their endpoints except for the case when an endpoint is $\pm\sqrt a$ or $0$.

Thus, the set $\Sigma_1$ in the plane $\bbR^2(h,k)$ contains curves defined on intervals from $\cI_{a,g}$, 
``separate'' points corresponding to points from $\Theta^0_{a,g}$, and, possibly, something else corresponding to $x=0$. 
An explicite description of $\Sigma_1$ is given in the following statement.

\begin{theorem} \label{Th:BifDiag} 
The set $\Sigma_1$ for the integrable Hamiltonian system with Hamiltonian $\eqref{Eq:H}$ and the integral $K=S_3$ on orbit $\eqref{Eq:LevelSurf}$ 
is the union of the following parametric curves and points on the plane $\bbR^2(h,k)${\rm:}

$1)$ 
the pairs of curves $(h_\pm(x),k_\pm(x))$, $x\in I$, for each $I\in\cI_{a,g}$, where
\begin{equation} \label{Eq:BifCurve}
\begin{aligned} 
&h_\pm(x)=\frac{(g{-}k_\pm(x)x)^2}{2(a-x^2)}{+}\frac{k_\pm^2(x)}{2\beta}\,{-}\,\frac{(a{-}x^2)g^2_1}2+g_2g+g_3k_\pm(x)+V,\!\!\!\\   
&k_\pm(x)=\frac{g(a+x^2)}{2ax}-\frac{(a-x^2)^2}{2ax}\frac{\partial g_3}{\partial x}\pm\frac{(a-x^2)}{2ax}\times\!\!\!\\
\!\!\!\!\times&\sqrt{\!\!\Bigl(\!g{-}(a{+}x^2)\frac{\partial g_3}{\partial x}\!\Bigr)^2\!\!\!{-}4ax
\Bigl(\!xg_1^2{-}(a{-}x^2)g_1\frac{\partial g_1}{\partial x}{+}g\frac{\partial g_2}{\partial x}{+}
x\Bigl(\!\frac{\partial g_3}{\partial x}\!\Bigr)^2\!\!{+}\frac{\partial V}{\partial x}\!\Bigr)};\!\!\!
\end{aligned} 
\end{equation} 

$2)$ 
the points $(h(x_0),k(x_0))$ for each $x_0\in\Theta^0_{a,g}$, where
\begin{equation*}
\begin{aligned} 
h&(x_0)=\frac{(g{-}k(x_0)x_0)^2}{2(a-x_0^2)}+\frac{k^2(x_0)}{2\beta}-\frac{(a{-}x_0^2)g^2_1}2+g_2g+g_3k(x_0)+V,\\   
k&(x_0)=\frac{g(a+x_0^2)}{2ax_0}-\frac{(a-x_0^2)^2}{2ax_0}\frac{\partial g_3}{\partial x}(a,x_0),
\end{aligned} 
\end{equation*} 
and $g_1,g_2,g_3,V$ in these formulas mean the values of the corresponding functions at the point $(a,x_0)${\rm;}

$3)$
for the orbits $M^4_{a,g}$, where $g\ne a\frac{\partial g_3}{\partial x}(a,0)$, the point $(h_0,k_0)$, where
$$
\begin{aligned} 
h_0&=\frac{g^2}{2a}+\frac{k_0^2}{2\beta}-\frac{ag_1^2(a,0)}2+g_2(a,0)g+g_3(a,0)k_0+V(a,0),\\
k_0&=\frac{ag_1(a,0)\frac{\partial g_1}{\partial x}(a,0)-g\frac{\partial g_2}{\partial x}(a,0)-
\frac{\partial V}{\partial x}(a,0)}{\frac{\partial g_3}{\partial x}(a,0)-\frac ga};
\end{aligned} 
$$

$4)$
for the orbits $M^4_{a,g}$, where $g=a\frac{\partial g_3}{\partial x}(a,0)$ and $a$ satisfies the relation
$$
ag_1(a,0)\frac{\partial g_1}{\partial x}(a,0)-a\frac{\partial g_3}{\partial x}(a,0)\frac{\partial g_2}{\partial x}(a,0)-\frac{\partial V}{\partial x}(a,0)=0,
$$ 
the parabola
$$
h=\frac{k^2}{2\beta}{+}g_3(a,0)k{+}\frac a2\Bigl(\!\frac{\partial g_3}{\partial x}(a,0)\!\Bigr)^2\!\!{-}\frac a2g_1(a,0){+}a\frac{\partial g_3}{\partial x}(a,0)g_2(a,0){+}V(a,0).
$$
\end{theorem}

\begin{proof} 
All formulas in cases 1)--4) follow from equations \eqref{Eq:envelope} and expression~\eqref{Eq:W}. The cases 1) and 2) correspond to solutions of quadratic 
equation~\eqref{Eq:k-quadratic} for each parameters $x$ from $\Theta_{a,g}$, but in the case 2), when $x\in\Theta^0_{a,g}$, the corresponding 
discriminant $D_{a,g}(x)$ vanishes, since $D_{a,g}$ is a continuous function on $(-\sqrt a,\sqrt a)$. 

The case 3) corresponds to $x=0$ in equation~\eqref{Eq:k-quadratic}. If $B'_{a,g}(0)=\frac{\partial g_3}{\partial x}(a,0)-\frac ga\ne0$, 
then $-C'_{a,g}(0)/B'_{a,g}(0)$ is the unique solution $k_0$ of linear equation~\eqref{Eq:k-quadratic} for $x=0$, and we obtain the  point $(h_0,k_0)$ 
in the case~3). Note that if $B'_{a,g}(0)\ne0$, then the discriminant $D_{a,g}(x)$ is positive on some interval $(-\varepsilon,\varepsilon)$ and 
there are two bifurcational curves {\rm\eqref{Eq:BifCurve}} defined on $(-\varepsilon,0)$ and $(0,\varepsilon)$ which tend to the point $(h_0,k_0)$ 
as $x\to0$ and form one smooth bifurcational curve glued from two curves at this point.

The case 4) also corresponds to $x=0$, but the conditions on $g$ and $a$ in the case 4) are equivalent to the conditions $B'_{a,g}(0)=C'_{a,g}(0)=0$, 
which imply that an arbitrary $k$ is a solution of~\eqref{Eq:k-quadratic} for $x=0$. Thus, we obtain the required parabola in the case 4).
\end{proof} 

Note that for arbitrary functions $g_1,g_2,g_3,V$ the behavior of bifurcational curves described in Theorem~\ref{Th:BifDiag} by explicit formulas 
can be fairly complicated. They can have many cusps, intersect one another or coincide on some their arcs. Some general properties concerning 
the behavior of bifurcational curves are described in the following statement.

\begin{corollary} 
$1)$ 
If $J\subset\Theta_{a,g}$ is an open interval such that $D_{a,g}|_J>0$, then the bifurcational curve $(h_\pm(x),k_\pm(x))$ defined on $J$ 
by formulas {\rm\eqref{Eq:BifCurve}} is a smooth parametric curve 
which is regular for all~$x$, where $\frac{dk_\pm}{dx}(x)\ne0$. 

$2)$
Exactly two arcs of the bifurcational curves described in the items $1)$ and $4)$ of Theorem~{\rm\ref{Th:BifDiag}} tend to infinity 
such that $h(k)\sim\frac{k^2}{2\beta}$ {\rm(}one arc for $k\to+\infty$ and one arc for $k\to-\infty${\rm)}. For the curves defined 
by formulas {\rm\eqref{Eq:BifCurve}} these arcs correspond to $x\to0$. 

$3)$
For each singular point $P_\pm$ of rank~$0$ which is of center-center type {\rm(}by Theorem~{\rm\ref{Th:Rank0Points}} there can be 
$0$, $1$, or $2$ such points{\rm)} there are exactly two arcs of the bifurcational curves described by formulas {\rm\eqref{Eq:BifCurve}}
which tend to the corresponding point~$Z_\pm$ described in Theorem~{\rm\ref{Th:BifDiag0}} as $x\to\pm\sqrt a$.
\end{corollary}

\begin{proof} 
Since $h=h_\pm(x)$, $k=k_\pm(x)$ satisfy equations \eqref{Eq:envelope}, we have
$$
\frac{dh_\pm}{dx}(x)=\frac{\partial W_{a,g}}{\partial k}(k_\pm(x),x)\frac{dk_\pm}{dx}(x). 
$$
Therefore, the parametric curve \eqref{Eq:BifCurve} is regular iff $\frac{dk_\pm}{dx}(x)\ne0$ and
can have singularities {\rm(}for example, cusps{\rm)} only at points, where $\frac{dk_\pm}{dx}=0$. 

Items 2) and 3) follow from formulas \eqref{Eq:BifCurve} by investigating the behavior of the parametric curves $(h_\pm(x),k_\pm(x))$ 
as $x$ tends to $0$ or $\pm\sqrt a$. Note that $D_{a,g}$ is positive in a neighborhood of the points $\pm\sqrt a$ iff
$q$ from Corollary \ref{Cor:Rank0Points} is positive for $R_3=\pm\sqrt a$.
\end{proof} 

\section{Liouville tori bifurcations}

All basic definitions and facts about Liouville tori bifurcations can be found in \cite{BF}.

\begin{theorem} \label{Th:Rank1Type}
A singular point of rank $1$ {\rm(}described in Theorem {\rm\ref{Th:CriticalCircle}} and Corollary {\rm\ref{Cor:CriticalCircle}}{\rm)}
is non-degenerate iff $\frac{\partial^2 W_{a,g}(k,x)}{\partial x^2}\ne0$, where $W_{a,g}(k,x)$ is given by {\rm\eqref{Eq:W}}. Moreover,
\begin{itemize}
\item if $\frac{\partial^2 W_{a,g}(k,x)}{\partial x^2}>0$, then the type of the point is elliptic{\rm;}
\item if $\frac{\partial^2 W_{a,g}(k,x)}{\partial x^2}<0$, then the type of the point is hyperbolic.
\end{itemize}
\end{theorem}

The non-degeneracy and the type of a singular point $y$ of rank $1$ are completely determined by the spectrum of linearization of the Hamiltonian vector field
which is a (non-trivial) linear combination of $\sgrad H$ and $\sgrad K$ vanishing at $y$. Thus, Theorem~\ref{Th:Rank1Type} follows from the following statement.

\begin{lemma} \label{L:SpecturmAF}
Each point $y$ of rank $1$ {\rm(}described in Theorem {\rm\ref{Th:CriticalCircle}} and Corollary~{\rm\ref{Cor:CriticalCircle}}{\rm)}
is a singular point for the vector field $\sgrad F_y$, where $F_y=H-\lambda K$ and $\lambda=\frac{\partial H}{\partial k}\big|_y$. 
The spectrum of the linearization $A_{F_y}=\mathrm{Lin}(\sgrad F_y)$ at the point $y$ consists of $4$ zeroes and 
$$
\mu_\pm=\pm i\sqrt{\frac{\partial^2 W_{a,g}(k,x)}{\partial x^2}}.
$$
\end{lemma}

\begin{proof} 
The proof is by direct calculation. The Hamiltonian vector fields $\sgrad H$ and $\sgrad K$ in the coordinates $(x,m,\phi,k,a,g)$ from Lemma~\ref{L:coord} 
are given by \eqref{Eq:sgradHandK}, and at a point $y\in\eLie^*$ of rank~1 conditions \eqref{Eq:lambda} are fulfilled. Hence for the function $F_y=H-\lambda K$, 
where $\lambda=\frac{\partial H}{\partial k}\big|_y$, we have $\sgrad_yF_y=0$, and therefore the linearization $A_{F_y}$ of the field 
$$
\sgrad F_y= 
\Bigl((a-x^2)\frac{\partial H}{\partial m},-(a-x^2)\frac{\partial H}{\partial x},\frac{\partial H}{\partial k}-\lambda,0,0,0\Bigr)
$$
at the point $y$ is well-defined. Taking into account conditions \eqref{Eq:lambda}, we get the following equation for the spectrum of $A_{F_y}$:
$$
\det(A_{F_y}-\mu\,\mathrm{Id})=\mu^4(a-x^2)^2\det
\left(\begin{matrix}   
\frac{\partial^2 H}{\partial m \partial x}-\mu&\frac{\partial^2 H}{\partial m^2}\\
-\frac{\partial^2 H}{\partial x^2}&-\frac{\partial^2 H}{\partial x\partial m}-\mu
\end{matrix}\right)=0.
$$
Thus the non-zero eigenvalues of $A_{F_y}$ are
\begin{equation} \label{Eq:EigenRank1}
\mu_\pm=\pm\sqrt{\Bigl(\frac{\partial^2 H}{\partial x\partial m}\Bigr)^2-\frac{\partial^2 H}{\partial x^2}\frac{\partial^2 H}{\partial m^2}}.
\end{equation}

For the function $H$ given by \eqref{Eq:HamiltIntegrCoord}  we have
\begin{equation} \label{Eq:SecondDer}
\begin{gathered}
\frac{\partial^2 H}{\partial m^2}=\frac1{a-x^2},\qquad\frac{\partial^2 H}{\partial x\partial m}=\frac{\partial g_1}{\partial x}(a,x)+\frac{2mx}{(a-x^2)^2},\\
\frac{\partial^2 H}{\partial x^2}=\frac{(g^2+ak^2+m^2)(a+3x^2)-2gkx(x^2+3a)}{(a-x^2)^3}+\\
+m\frac{\partial^2 g_1}{\partial x^2}(a,x)+g\frac{\partial^2g_2}{\partial x^2}(a,x)+k\frac{\partial^2 g_3}{\partial x^2}(a,x)+\frac{\partial^2 V}{\partial x^2}(a,x).\\
\end{gathered}
\end{equation}
Since, by Theorem \ref{Th:CriticalCircle}, at a singular point we have $m=-(a-x^2)g_1(a,x)$, equalities \eqref{Eq:SecondDer} can be rewritten as
\begin{equation} \label{Eq:SecondDer2}
\begin{gathered}
\frac{\partial^2 H}{\partial m^2}=\frac1{a-x^2},\qquad\frac{\partial^2 H}{\partial x\partial m}=\frac{\partial g_1}{\partial x}(a,x)-\frac{2xg_1(a,x)}{a-x^2},\\
\frac{\partial^2 H}{\partial x^2}=\frac{\partial^2 W_{a,g}(k,x)}{\partial x^2}+(a-x^2)\Bigl(\frac{\partial g_1}{\partial x}(a,x)-\frac{2xg_1(a,x)}{a-x^2}\Bigr)^2,\\
\end{gathered}
\end{equation}
where $W_{a,g}(k,x)$ is given by \eqref{Eq:W}.
Substituting expressions \eqref{Eq:SecondDer2} into formula \eqref{Eq:EigenRank1} we get the desired expression for $\mu_\pm$.

Lemma \ref{L:SpecturmAF} and, consequently, Theorem \ref{Th:Rank1Type} are proved.
\end{proof}

\begin{theorem} \label{Th:Atoms}
The only possible non-degenerate Liouville tori bifurcations for the isoenergy surfaces $Q^3$ of the integrable Hamiltonian system with 
Hamiltonian $\eqref{Eq:H}$ and the integral $K=S_3$ on orbit $\eqref{Eq:LevelSurf}$ are the so-called $A$ and $V_k$ bifurcations. 
In particular, if there is only one singular circle in a fiber, then the bifurcation is either $A$ or $B$.
\end{theorem}

\begin{proof} 
There is only one elliptic bifurcation (of type $A$), thus we consider hyperbolic bifurcations.
Since all critical points of rank $1$ satisfy the condition $R_1^2 + R_2^2 \not =0$, we can work in the coordinates $(x,m,\phi,k,a,g)$. 

Consider the inverse image of a point $(h_0,k_0)$ under the momentum mapping $M^4_{a,g}\to\bbR^2(h,k)$. 
Then $\phi$ is arbitrary and $m$ is given by 
\begin{equation} \label{Eq:InvImm}\
\frac{(m+(a-x^2)g_1(a,x))^2}{2(a-x^2)}=h_0-W_{a,g}(k_0,x),
\end{equation} 
where $x$ satisfies the condition $h_0\geq W_{a,g}(k_0,x)$.

Thus any connected component of a singular fiber for a non-degenerate singularity is a product of $S^1$ and a wedge sum of $k$ circles as in Figure~\ref{Fig:AtomVk}. \begin{figure} 
    \centering
    \includegraphics[width=\textwidth]{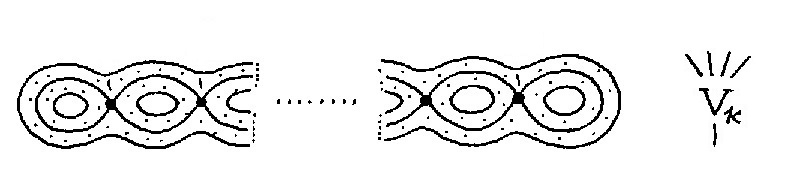}
    \caption{Atom $V_k$.}
    \label{Fig:AtomVk}
\end{figure} More precisely, the set in the plane $(m,x)$ given by equation \eqref{Eq:InvImm} is homeomorphic to the union of circles that are joined at the points $h_0=W_{a,g}(k_0,x)$.

Since the singularity is non-degenerate, this is precisely the bifurcation for the $V_k$ atom. Theorem \ref{Th:Atoms} is proved.
\end{proof}

\section{Isoenergy surfaces}

For a Hamiltonian function $H$ on $\eLie^*$ which is a positive definite quadratic form in $\bS$, the topology of isoenergy surfaces is completely 
determined by their projections on the Poisson shere (for details see \cite{BF}). By Theorem \ref{Cor:CriticalCircle}, the projection is invariant under rotation 
around the $R_3$-axis. As a direct consequence we get the following statement.

\begin{theorem} \label{Th:Q3}
Any isoenergy surface $Q^3$ of the integrable Hamiltonian system with Hamiltonian $\eqref{Eq:H}$ and the integral $K=S_3$ on orbit $\eqref{Eq:LevelSurf}$ 
is either $\mathbb{RP}^3$ or a disjoint union of $k$ products $S^1\times S^2$ and not more than two spheres $S^3$. 
\end{theorem} 

\begin{proof}  
If the projection of $Q^3$ on the Poisson sphere is surjective, then $Q^3=\mathbb{RP}^3$.
Otherwise the image of the projection is the unioun of $l$ rings and not more than two disks with centers in the poles $\bR=(0,0,R_3)$. 
Each ring corresponds to $S^1\times S^2$ and each disk to $S^3$. 
\end{proof}

\paragraph{ACKNOWLEDGMENTS} This work was supported by the Russian  Science Foundation, project no. 17-11-01303.


\end{document}